\documentclass[twocolumn,amsthm]{autart}    
\usepackage{amsmath,amssymb}
\usepackage{xcolor}

\usepackage{accents}
\usepackage[toc,page]{appendix}
\usepackage{enumerate}
\usepackage{graphicx,floatrow} 

\newcommand{\bv}[1]{{\mathbf #1}} 
\newcommand{\bt}[1]{{\mathbf{#1}}}

\newcommand{\funSpace}[1]{\mathcal{#1}} 

\newcommand{\rom}[1]{{{#1}_{\rm r}} }

\newcommand{\R}{\mathbb{R}} 
\newcommand{\C}{\mathbb{C}} 

\newcommand{\mexp}[1]{{e^{#1}}}

\newcommand{\acro}[1]{{\texttt{#1}}}  

\newcommand{\uss}{w}
\newcommand{\xiss}{{\boldsymbol{\xi}}}
\newcommand{\Ass}{\bt A_{\rm \xi}}
\newcommand{\css}{\bt c_{\rm \xi}}
\newcommand{\Pss}{\boldsymbol{\Pi}} 

\newcommand{\basVec}{g} 
\newcommand{\iu}{{\mathfrak{i}\mkern1mu}} 
\newcommand{\errom}{\mathit{E}} 

\newcommand{\bunderline}[1]{{\breve{#1}}} 

\newcommand{\xer}{\bunderline{\bv{x}}}
\newcommand{\Aer}{\bunderline{\bt{A}}}
\newcommand{\ber}{\bunderline{\bv{b}}}
\newcommand{\cer}{\bunderline{\bv{c}}}
\newcommand{\der}{\bunderline{d}}
\newcommand{\yer}{\bunderline{{y}}}
\newcommand{\uer}{\bunderline{{u}}}
\newcommand{\Qer}{\bunderline{{\bt{Q}}}}
\newcommand{\Fss}{\bunderline{\mathcal{F}}} 

\begin{document}

\begin{frontmatter}

\title{A time domain a posteriori error bound for balancing-related model order reduction\thanksref{footnoteinfo}} 

\thanks[footnoteinfo]{
Corresponding author B.~Liljegren-Sailer. Tel. +49 651 201-3468.
}

\author[Trier]{Bj{\"o}rn Liljegren-Sailer}\ead{bjoern.sailer@uni-trier.de} 

\address[Trier]{Universit{\"a}t Trier, FB IV - Mathematik, Lehrstuhl Modellierung und Numerik, D-54286 Trier, Germany}  

    
\begin{abstract}
The aim in model order reduction is to approximate an input-output map described by a large-scale dynamical system with a low-dimensional and cheaper-to-evaluate reduced order model. While high fidelity can be achieved by a variety of methods, only a few of them allow for rigorous error control. In this paper, we propose a rigorous error bound for the reduction of linear systems with balancing-related reduction methods. More specifically, we consider the simulation over a finite time interval and provide an a posteriori adaption of the standard a priori bound for Balanced Truncation and Balanced Singular Perturbation Approximation in that setting, which improves the error estimation while still yielding a rigorous bound. Our result is based on an error splitting induced by a Fourier series approximation of the input and a subsequent refined error analysis. We make use of system-theoretic concepts, such as the notion of signal generator driven systems, steady-states and observability. Our bound is also applicable in the presence of nonzero initial conditions. Numerical evidence for the sharpness of the bound is given.
\end{abstract}
\end{frontmatter}

\section{Introduction}\label{sec1}
Consider the linear time-invariant system
\begin{align}
\begin{aligned}\label{eq:fom}
 \dot{\bv x} (t) &= \bt A \bv x(t) + \bv b  u(t), \hspace{0.3cm} 
 \hspace{0.5cm} \bv x(0) = \bv x_0 \in \R^N,\\
   y(t) & = \bv c \bv x(t) + d u(t)
\end{aligned}  
\end{align}
for $t \in [0,T]$ with state $\bv x:[0,T]\rightarrow \R^N$, initial conditions $\bv x_0 $ and an asymptotically stable state equation, i.e., $\bt A\in \R^{N,N}$ Hurwitz. Moreover, let $\bv b \in \R^{N,1}$, $\bv c \in \R^{1,N}$ and $d\in \R$. When the state dimension $N$ is large compared to the dimension of the input $u$ and output $y$ -- for ease of presentation we assume the scalar-valued case $u,y:  [0,T]\rightarrow \R$ -- the computational costs for evaluating \eqref{eq:fom} many times can become very high. This model is referred to as full order model (\acro{FOM}). We are interested in a lower-dimensional reduced order model (\acro{ROM}) that is cheaper to evaluate and still sufficiently accurate, in the sense that it reproduces a similar output $\rom{y} \approx y$ for the inputs $u$ of interest. The considered \acro{ROM} reads 
\begin{align}
\begin{aligned}\label{eq:rom}
 \rom{\dot{\bv{x}}} (t) &= \rom{\bt{A}} \rom{\bv{x}}(t)+ \rom{\bv b}  u(t), \hspace{0.5cm} \rom{\bv x}(0) = \rom{\bv{x}}_0 \in \R^n,\\
   \rom{y}(t) &= \rom{\bv c} \rom{\bv x}(t) + \rom{d} u(t),
\end{aligned}
\end{align}
with reduced state $\rom{\bv x}:[0,T]\rightarrow \R^n$, $n\ll N$. A variety of model order reduction methods (\acro{MOR}) have been proposed, see e.g., \cite{book:antoulas2005,book:dimred2003}. For example, in the projection-based approaches one seeks for appropriate reduction bases $\bt V, \bt W \in \mathbb{R}^{N,n}$ with $\bt W^T \bt V = \bt I_n$ (unit matrix). The \acro{ROM} is then defined by $ \rom{\bv{x}}_0= \bt W^T \bv{x}_0$, $\rom{\bt A} = \bt W^T \bt A \bt V$, $\rom{\bt{b}} = \bt W^T \bv b$, $\rom{\bt{c}} =  \bv c \bt V$ and  $\rom{d}= d$. The reduction error
\begin{align*}
\errom(t) :=  y(t)-\rom{y}(t), \hspace{1cm} t \in [0,T],
\end{align*}
is not known in practice, so it is of high interest to have bounds or estimates on it.  For its analysis it is useful to introduce the error system
\begin{align}
\begin{aligned} \label{eq:errSys}
\dot{\xer}(t) &= \underbrace{\begin{bmatrix}
		\bt A & \\
		 & \rom{\bt{A}}
	\end{bmatrix}}_{=: \Aer} \xer(t) +
	\underbrace{ \begin{bmatrix} \bv b \\
	\rom{\bv{b}}
	\end{bmatrix}}_{=: \ber } \uer(t), \quad \xer(0) = \xer_0,\\
	{\yer(t)} &= \underbrace{\begin{bmatrix} \bv c ,&
	-\rom{\bv{c}} 
	\end{bmatrix}}_{=:\cer} \xer(t) + 
	\underbrace{(d-\rom{d})}_{=:\der} \uer(t).
\end{aligned}
\end{align}
By construction, $\yer(t) =\errom(t)$ for $t \in [0,T]$ if $\uer =u$ and $\xer_0 = [\bv x_0^T, \rom{\bv x}_0^T]^T$ are chosen. Thus, \eqref{eq:errSys} is just a concise representation of the error dynamics.

In this paper, a rigorous a posteriori bound for the \mbox{$\funSpace{L}^2$-error} in finite time $[0,T]$ is derived from system-theoretic concepts. It is most naturally applied in combination with balancing-related \acro{MOR} methods, such as Balanced Truncation (\acro{BT}) and Balanced Singular Perturbation Approximation (\acro{SPA}) \cite{book:green2012linear,art:BT-moore1981,art:SPA-liu1989}, since it exploits and relies on the key features of these methods, summarized in the following assumption.
\begin{assum} \label{assum:RomRequired}
The \acro{FOM} is considered for square-integrable input $u:[0,T] \rightarrow \R$ and reduced by a method, for which the following holds:
\begin{enumerate}[i$_)$]
\item \label{cond:balrel-i} An a priori bound is accessible for the \acro{ROM}. That is, assuming zero initial conditions ($\bv x_0= \bv 0$), it holds 
$$
\| \errom \|_{\funSpace{L}^2_T} :=  \sqrt{\int_{0}^T \text{\textbar} \errom(t) \text{\textbar}^2 dt } 
\, \leq   \alpha   \|u \|_{\funSpace{L}^2_T} 
$$
for a known constant $\alpha>0$. 
\item The observability Gramian $\bt Q$ of the \acro{FOM} is required for the construction of the \acro{ROM}. \label{cond:balrel-ii}
\item The \acro{ROM} is asymptotically stable. \label{cond:balrel-iii}
\end{enumerate}
\end{assum}
Here and in the following, $\|\boldsymbol{\cdot} \|_{\funSpace{L}^2_T}$ refers to the $\funSpace{L}^2$-norm on the time-interval $[0,T]$. The availability of an \textit{a priori bound}  as in Assumption~\ref{assum:RomRequired} is of high practical value on its own, especially for the selection of an appropriate \acro{ROM} dimension $n$. But this a priori result yields a worst-case estimation on the error, since it does not take into account any specific knowledge on the simulation setup. The bound we propose in this paper can be considered an \textit{a posteriori} adaption of the a priori bound. The main idea is to filter out a signal $\uss$ and to apply the pessimistic a priori bound to the remainder $u- \uss$ only. The approximation $\uss$ is chosen such that a more explicit error analysis becomes feasible that leads to less overestimation. We further split the error part related to $\uss$ into its long-time behavior, also denoted as \textit{steady-state}, and a term that is equivalent to the output energy a certain initial state has in the error system \eqref{eq:errSys}. The steady-state can be determined in an offline-online efficient way, and the other error term can be sharply bounded using the approach from \cite{art:bls-InhomInit}.


The peculiarity of our proposed error bound is that it takes into account information of the input $u$, i.e., is an a posteriori result, and, simultaneously, exploits the features inherent in balancing-related \acro{MOR}, cf., Assumption~\ref{assum:RomRequired}. Consequently, it outperforms
other rigorous a posteriori bounds, e.g., the residual based bounds derived in \cite{haasdonk2011efficient}. Let us mention that effective error estimates have been proposed in literature \cite{feng2017some,feng2019new}, but those are typically not rigorous and thus may lead to an underestimation of the actual error, which is problematic for certain applications.

The structure of the paper revolves around the derivation of our error bound. In Section~\ref{sec:BalRedResults}, we briefly outline the standard a priori bound for balancing-related \acro{MOR}, and the error bound from \cite{art:bls-InhomInit} for inhomogeneous initial conditions. Those are ingredients of our bound. Additionally, we make use of the notion of steady-states for a certain class of approximations on the input in our error analysis; the respective results are stated in \mbox{Sections~\ref{sec:steadystate}-\ref{sec:steadyExp}.} Our a posteriori bound is proven in Section~\ref{sec:MainResultErr}. Its effectivity is numerically demonstrated in Section~\ref{sec:NumResults} at the example of two academic benchmarks.


\section{Results for balancing-related MOR} \label{sec:BalRedResults}

The observability and controllability of a state in the \acro{FOM} can be quantified in terms of the observability Gramian~$\bt{Q}$ and controllability Gramian $\bt P$, defined as
\begin{align*}
\bt{Q} = \int_{0}^\infty \mexp{t \bt A^T} \bv c^T \bv c \mexp{t \bt A} dt, \quad
\bt{P} = \int_{0}^\infty \mexp{t \bt A} \bv b \bv b^T \mexp{t \bt A^T} dt.
\end{align*}
These matrices are well-defined for asymptotically stable systems. They play a crucial role in the analysis and implementation of \acro{BT} and \acro{SPA}, see \cite{book:antoulas2005}.  Notably, these balancing-related \acro{MOR} methods allow for a rigorous error analysis that is based on the Gramians. In this section, we outline the error bounds from literature which also play a role in our refined error analysis. These results require, respectively justify, Assumption~\ref{assum:RomRequired} in our approach.

\subsection{A priori error bound}
Asymptotically stable systems can be transformed in a so-called balanced form, in which the Gramians $\bt Q$ and $\bt P$ are simultaneously diagonalized and equal. The diagonal entries  $\sigma_1 \geq \sigma_2 \geq \ldots \geq \sigma_N \geq 0$ of the resulting diagonal matrix are called the Hankel Singular Values (\acro{HSV}s) and represent a measure for observability and controllability of the balanced states. In both \acro{BT} and \acro{SPA}, the \acro{ROM} is composed of the states related to the largest \acro{HSV}s, see \cite{book:green2012linear} for details on the methods. Given the \acro{ROM} dimension $n$ is chosen such that $\sigma_n > \sigma_{n+1}$ and zero initial conditions are considered ($\bv x_0 = \bv 0$), the reduction error can be shown to fulfill the a priori error bound $\| \errom \|_{\funSpace{L}^2_\infty} \leq  \alpha   \|u \|_{\funSpace{L}^2_\infty}$ with $\alpha = 2 \sum_{j = n+1}^{N} \sigma_j$ composed of the neglected \acro{HSV}s, see \cite{book:dimred2003}. As we require a bound on a finite time interval $[0,T]$ (Assumption~\ref{assum:RomRequired}-\ref{cond:balrel-i}$_)$), we apply the former bound on the extension of the input $u: [0,T]\rightarrow \R$ to the infinite interval $[0,\infty)$, obtained by setting $u(t) = 0$ for $t>T$. This yields
\begin{align*} 
\| \errom \|_{\funSpace{L}^2_T} \leq  \| \errom \|_{\funSpace{L}^2_\infty} \leq \alpha   \|u \|_{\funSpace{L}^2_T}.
\end{align*}
\subsection{Error originating from initial conditions}
The case with trivial input ($u(t)=0$ for $t\geq 0$) allows for a simple error analysis based on the notion of observability.
In this setting, the output of the error system \eqref{eq:errSys} is solely determined by the initial conditions $\xer(0) =\xer_{0}= [\bv x_{0}^T,\rom{\bv x}_{0}^T]^T$. It reads $y_{\xer_{0}} =\cer \mexp{t\Aer} \xer_0$. As shown in \cite{art:bls-InhomInit}, a bound for its norm follows directly from applying the notion of observability to the error systems, i.e.,
\begin{align}
 \| y_{\xer_{0}}\|_{\funSpace{L}^2_T} &\leq   \| y_{\xer_{0}}\|_{\funSpace{L}^2_\infty}=
  \sqrt{
\xer_{0}^T \Qer
\xer_{0}
}, \label{eq:erInitCond}
\\
&\Qer = \int_{0}^\infty \mexp{t \Aer^T} \cer^T \cer \mexp{t \Aer} dt =
\begin{bmatrix}
 \bt{Q} & \bt{S} \\
 \bt{S}^T & \rom{\bt{Q}}
\end{bmatrix}. \nonumber
\end{align}
Hereby, $\Qer$ is the observability Gramian of the error system, composed of the Gramians $\bt{Q}$ and $\rom{\bt{Q}}$ of the \acro{FOM} and \acro{ROM} and a matrix $\bt{S}\in \R^{N,n}$ that can be determined by solving a sparse-dense Sylvester equation \cite{art:bls-InhomInit}. Note that $\bt{Q}$ is already required for constructing the \acro{ROM} with \acro{BT} or \acro{SPA} and thus is available for the error bound without additional costs (Assumption~\ref{assum:RomRequired}-\ref{cond:balrel-ii}$_)$). Only the matrices $\rom{\bt{Q}}$ and $\bt{S}$ have to be determined, and this has a comparably low computational cost. Once this is done, the evaluation of \eqref{eq:erInitCond} for any initial condition only requires matrix-vector multiplications, and thus can be used as an online-efficient error bound.
\begin{rem} \label{rem:timeLimited}
The inequality \eqref{eq:erInitCond} holds for any $T>0$ and becomes an equality for $T \to \infty$. On the other hand, when using time-limited model order reduction, e.g., \cite{art:redmann2020lt2}, $\Qer$ can be replaced by the time-limited Gramian, and the inequality then also becomes an equality for a finite time $T$. 
\end{rem}

\begin{rem}\label{rem:largeScale}
In a large-scale setting it may be necessary to substitute the Gramians with low-rank approximations for computational reasons, see \cite{book:dimred2003}. In the presence of low-rank errors the results of this section cannot be shown rigorously, but they can still serve as a basis for error estimation.
\end{rem}

\section{Linear signal generator driven systems} \label{sec:steadystate}
Consider the  linear asymptotically stable system
\begin{align}
\begin{aligned} \label{eq:genersys}
	\dot{\xer}(t) &= \Aer \xer(t) + \ber \uss(t), \hspace{0.5cm} \xer(0) = \xer_0,\\
	\yer(t) &=  \cer \xer(t) + \der \uss(t),
\end{aligned}
\end{align}
with system matrices as in \eqref{eq:errSys}. We assume it to be driven by a linear signal generator, i.e., $\uss$ to be described a linear autonomous differential equation
\begin{align}
\uss(t) = \css \xiss(t), \hspace{0.8em} \dot{\xiss}(t) = \Ass \xiss(t), \hspace{0.8em} \xiss(0) = \xiss_0 \in \C^{q}, \label{eq:siggen}
\end{align}
with $\Ass \in \C^{q,q}$ and $\css \in \C^{1,q}$. We employ the notion of steady-states from \cite{art:steadystate-mor-2010,art:isidori2008steady} and the representation of the solution to \eqref{eq:genersys} induced by it. Steady-state refers in this context to the long-time behavior the solution approaches independently of the choice of initial condition $\xer_0$. Note that the impact of the initial condition fades away over time due to the asymptotic stability of the system. We assume $\Aer$ and $\Ass$ to not have any common eigenvalues, which implies that the Sylvester equation
\begin{align} \label{eq:sylveq}
 \Aer \Pss + \ber \css = \Pss \Ass
\end{align}
is uniquely solvable for $\Pss \in \R^{N+n, q}$. It follows that
\begin{align*}
 \begin{bmatrix}
 \Aer & \ber \css \\
 \bt 0 & \Ass
 \end{bmatrix}
  \begin{bmatrix}
 \Pss\\
 \bt{I}_q
 \end{bmatrix}
 &= 
 \begin{bmatrix}
 \Aer \Pss + ( -\Aer \Pss + \Pss \Ass) \\
  \Ass
 \end{bmatrix}
  =   \begin{bmatrix}
 \Pss\\
\bt I_q
 \end{bmatrix} \Ass,
\end{align*}
with $\bt I_q \in \R^{q,q}$ denoting the unit matrix. Using the latter relation, it can be shown that the steady-state of the signal generator driven system reads $\xer_{\rm st}(t) = \Pss \xiss(t)$, $t\geq 0$. This is the specific solution of \eqref{eq:genersys}-\eqref{eq:siggen} with $\xer_0 = \Pss \xiss_0$, since
\begin{align*}
 \frac{d}{dt} 
 \begin{bmatrix}
 	\xer_{\rm st}(t) \\
 	\xiss(t) 
 \end{bmatrix}
 & = 
\begin{bmatrix}
 \Aer & \ber \css \\
 \bt 0 & \Ass
 \end{bmatrix}
  \begin{bmatrix}
 	\Pss \\
 	\bt{I}_q
 \end{bmatrix}\xiss(t) \\
&= 
\begin{bmatrix}
 \Pss\\
\bt I_q
 \end{bmatrix} \Ass \xiss(t) = \begin{bmatrix}
 \Pss\\
\bt I_q
 \end{bmatrix} \frac{d}{dt} \xiss(t).
\end{align*}
Thus, we can assign exactly one steady-state to a signal $\uss$ given by a linear signal generator, and the linear mapping
\begin{align*}
	\Fss: \uss(\boldsymbol{\cdot}) \mapsto \Fss(\uss(\boldsymbol{\cdot})) = (\cer \Pss +\der \css) \xiss(\boldsymbol{\cdot})
\end{align*}
is well-defined. Finally note that the output response of \eqref{eq:genersys}-\eqref{eq:siggen} with a general choice of $\xer_0$ reads
\begin{align*}
 \yer(t) 
 &= \underbrace{(\cer \Pss  + \der \css) \xiss(t)}_{\text{steady-state response $\Fss(\uss)$}} + \quad \underbrace{\cer e^{\Aer t} (\xer_0 - \Pss \xiss_0 )}_{\text{transient response}},
\end{align*}
and the transient response decays exponentially. An illustration of the convergence to the steady-state is shown in Fig.~\ref{fig:beam-steadyst}, using zero initial conditions. 
\begin{figure}[htb]
\floatbox[{\capbeside\thisfloatsetup{capbesideposition={right,top},capbesidewidth=2.8cm}}]{figure}[\FBwidth]
{\caption{Illustration of the convergence to the steady-\mbox{state} response for signal $\uss(t) = \cos(5t)$ and the model given by the Beam benchmark (cf.,\,Section~\ref{subsec:numBeam} and \cite[Section~24]{book:dimred2003}).
\label{fig:beam-steadyst}
}}
{\includegraphics[height=3.5cm]{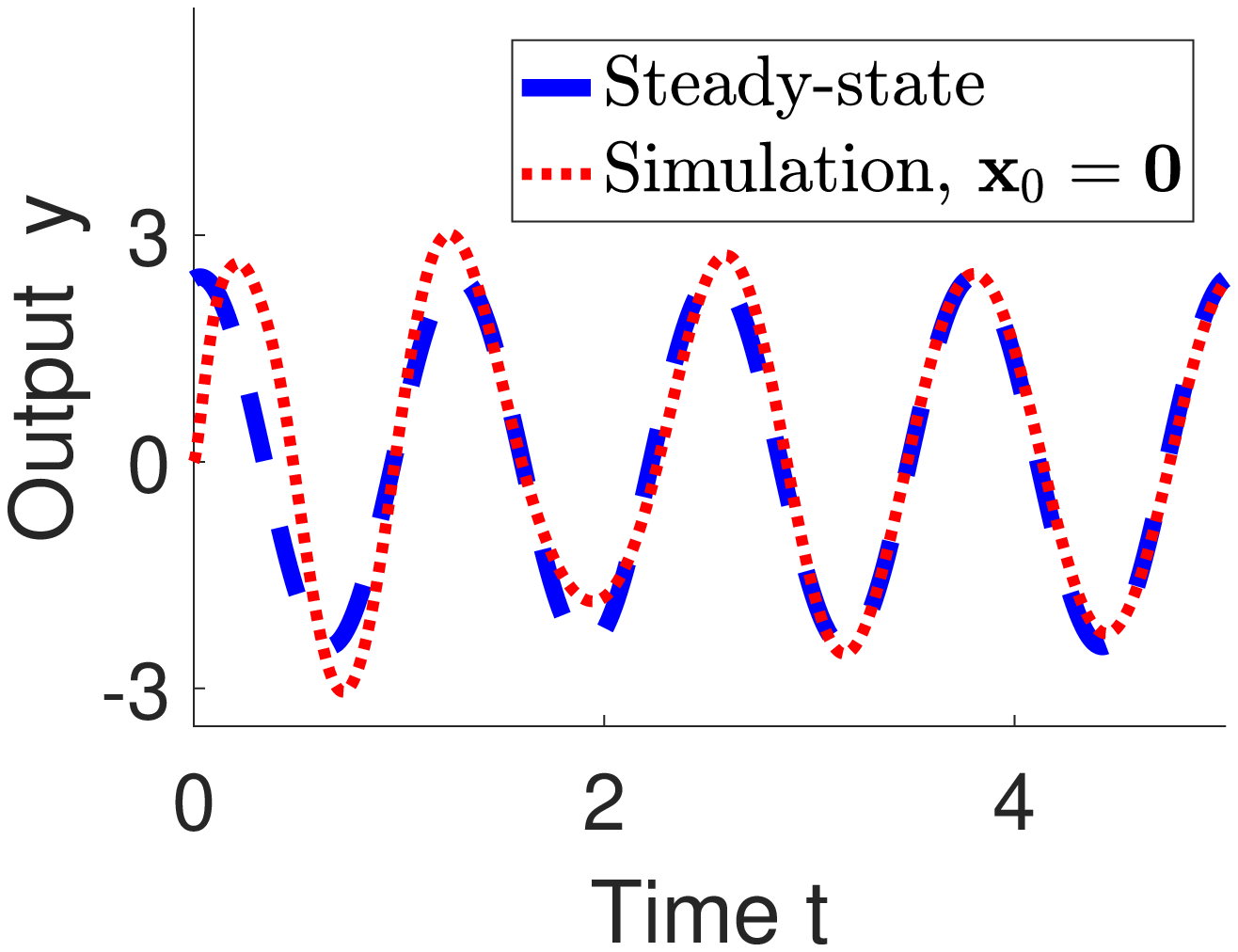}}                    
\end{figure}
\section{Steady-state response to a Fourier series} \label{sec:steadyExp}

For a prescribed order $K\geq 0$, the truncated Fourier series of the function $u: [0,T] \rightarrow \R$ is defined by
\begin{align}
	\uss^K(t) &= \lambda_0 + \sum_{\ell=1}^K \,\lambda_{\ell} {\cos\left(2\pi \ell \frac{t}{T} \right)}+ \lambda_{K+\ell} \,{\sin\left( 2\pi \ell \frac{t}{T}\right)} \nonumber \\
	&=: \sum_{\ell=0}^{2K} \,\lambda_{\ell} \basVec_\ell(t),  \quad t \in [0,T],  \label{eq:FourSig}
\end{align}
with Fourier coefficients $\lambda_\ell =  \int_0^T \basVec_\ell(t) u(t) dt/{\|\basVec_\ell\|_{\funSpace{L}_2^T}^2}$, $\ell = 0, \ldots, 2K$.
The truncated Fourier series $\uss^K$ and the steady-state response $\Fss(\uss^K)$ it infers on a linear system, cf.~Section \ref{sec:steadystate}, are examined in more detail in this section. A computationally convenient representation of $\Fss(\uss^K)$ is key for an efficient implementation of our error bound.

Let $\iu$ denote the imaginary unit ($\iu^2=-1$), and let $\alpha \in \mathbb{R}$. An exponential pulse $t\mapsto e^{\iu \alpha t}$ can be described by a signal generator as in \eqref{eq:siggen}, whereby a scalar-valued state \mbox{$\xi_\alpha: [0,T]\mapsto \C$} can be used. The related Sylvester equation \eqref{eq:sylveq} simplifies to a linear equation, i.e.,
\begin{align}
\begin{aligned}\label{eq:Pss_alpha}
		e^{\iu \alpha t} &= \xi_\alpha(t), \qquad  \dot{\xi}_\alpha(t) = \iu \alpha \xi_\alpha(t), \quad  \xi_\alpha(0) = 1,  \\
		\Pss_\alpha &= (\iu \alpha \bt I - \Aer)^{-1} \ber \in \C^{N}.
\end{aligned}
\end{align}
The steady-state response of this pulse is given by $\Fss( e^{\iu \alpha \boldsymbol{\cdot}}) =  (\cer \Pss_\alpha+ \der) e^{\iu \alpha \boldsymbol{\cdot}}$. Using that the cosine and sine functions are the real and imaginary parts of the exponential pulse, i.e., for any $t \geq 0$ it holds
\begin{align*}
	\cos( \alpha t) = \mathrm{Re}(e^{\iu \alpha t}), \quad 	\sin( \alpha t) = \mathrm{Im}(e^{\iu \alpha t}),
\end{align*}
we can derive a convenient representation of $\Fss(\basVec_\ell)$ for $\ell = 1, \ldots, 2K$. For the cosine functions that is
\vspace{-0.8cm}
{\small
\begin{align*}
	\Fss(\cos( \alpha \,\boldsymbol{\cdot})) &=  \cer \mathrm{Re} \left( \Pss_\alpha \bigl( \mathrm{Re}(e^{\iu \alpha \boldsymbol{\cdot}}) + \iu \, \mathrm{Im} ( e^{\iu \alpha \boldsymbol{\cdot}}) \bigr ) \right) + \der \cos( \alpha \,\boldsymbol{\cdot}) \\ 
	&=	(\cer \mathrm{Re}(\Pss_\alpha)+ \der) \cos(\alpha \,\boldsymbol{\cdot}) -  (\cer \mathrm{Im}(\Pss_\alpha)) \sin(\alpha \,\boldsymbol{\cdot}),
\end{align*}
}
and, by a similar calculation, it follows
\vspace{-0.8cm}
{\small
$$
\Fss(\sin( \alpha \,\boldsymbol{\cdot})) =  (\cer \mathrm{Im}(\Pss_\alpha)) \cos(\alpha \,\boldsymbol{\cdot}) +  (\cer \mathrm{Re}(\Pss_\alpha) + \der ) \sin(\alpha \,\boldsymbol{\cdot}).
$$
}
All in all, by employing the linearity of $\Fss$ and exploiting that  $\{\basVec_0, \ldots, \basVec_{2K}\}$ and $\{\Fss(\basVec_0), \ldots, \Fss(\basVec_{2K})\}$ are orthogonal sets of functions with respect to the $\funSpace{L}^2_T$-scalar product, we conclude the following lemma.
\begin{lem}\label{lem:SteadyFourier}
The steady-state response of system \eqref{eq:genersys} with $\uss=\uss^K$ as in \eqref{eq:FourSig}, i.e. $\Fss(\uss^K) =\sum_{\ell=0}^{2K} \lambda_\ell \Fss(\basVec_\ell)$, fulfills
\begin{align*}
(\| \Fss(\uss^K) \|_{\funSpace{L}^2_T})^2 &= 
T  \left| \cer \Pss_0 + \der \right|^2 \lambda_0^2 \\
&\quad + \frac{T}{2} \sum_{\ell=1}^{K}   \left| \cer \Pss_\ell +\der \right|^2 (\lambda_\ell^2+\lambda_{K+\ell}^2),
\end{align*}
whereby $\Pss_\ell$ for $\ell = 0,\ldots,K$ is given by \eqref{eq:Pss_alpha}. Further, the steady-state has the initial conditions
\begin{align*}
\xer_{{\rm st},0} := \lambda_0 \Pss_0 + \sum_{\ell=1}^K \lambda_\ell \mathrm{Re}(\Pss_\ell)  + \lambda_{K+ \ell} \mathrm{Im}(\Pss_\ell).
\end{align*}
\end{lem}

\begin{rem}
The quantities $\cer \Pss_\ell$, $\ell = 0, \ldots ,K$, are the so-callled moments of the error system \eqref{eq:errSys} at the frequencies $s_\ell = \iu \ell$, cf. \cite{book:antoulas2005}. In other words, they are the differences in the moments of the \acro{FOM} and the \acro{ROM}.
\end{rem}

\section{A posteriori error bound} \label{sec:MainResultErr}
The main result of this paper relies on the Fourier series approximation $\uss^K$ of the given input $u$, and a specific splitting of the reduction error. The error is decomposed into the steady-state and transient response induced by $\uss^K$ and the initial conditions, and the response originating from the remainder of the input $u-\uss^K$. Different techniques are used to bound these three terms individually.
\begin{thm}[A posteriori error bound]\label{thrm:aposteriori}
Let Assumption~\ref{assum:RomRequired} hold, considering the square-integrable input \mbox{$u:[0,T]\to \R$} and initial conditions $\bv x_0$ and $\rom{\bv x}_0$. Let $\Qer$ be as in \eqref{eq:erInitCond} and $K \in \mathbb{N}$. Let $\uss^K$, $ \Fss(\uss^K)$ and $\xer_{{\rm st},0}$ be as in Lemma~\ref{lem:SteadyFourier}.\\ Then the reduction error $\errom = y - \rom{y }$ is bounded by
\begin{align*}
\| \errom \|_{\funSpace{L}^2_T} \leq  \gamma_{K,u, \bv x_0} \hspace{5.4cm}\\
\gamma_{K,u, \bv x_0} = \| \Fss(\uss^K) \|_{\funSpace{L}^2_T}   + \sqrt{
 \xer_{c}^T  \Qer
\xer_{c}}
+ \, \alpha   \|u- \uss^K \|_{\funSpace{L}^2_T},
\end{align*}
where $\xer_{c} =[\bv{x}_{0}^T, \rom{\bv{x}}_{0}^T]^T- \xer_{{\rm st},0} \in \R^{N+n}$.
\end{thm}
\begin{proof}
We employ the linearity of the systems to split the reduction error into three sub parts. Each of them has a representation as output of the error system \eqref{eq:errSys} with a certain choice of input  $\uer$ and initial condition $\xer(0)$. We split the reduction error according to \mbox{$\errom= \yer_{\rm st}+  \yer_{\xer_c} + \yer_{\rm rest}$}, with
\begin{itemize}
\item $\yer_{\rm st}$ obtained by input $\uer= \uss$ and $\xer(0) = \xer_{{\rm st},0}$;
\item $\yer_{\xer_c}$ obtained for trivial input $\uer \equiv0$ and $\xer(0) = \xer_{c}$;
\item $\yer_{\rm rest}$ obtained by input $\uer = u-\uss^K$ and $\xer(0) = \bv 0$.
\end{itemize}
Clearly, $\| \errom \|_{\funSpace{L}^2_T} \leq \| \yer_{\rm st} \|_{\funSpace{L}^2_T} +  \| \yer_{\xer_c} \|_{\funSpace{L}^2_T} + \| \yer_{\rm rest} \|_{\funSpace{L}^2_T} $ holds by the triangle inequality. The claimed error bound $\gamma_{K,u, \bv x_0}$ is obtained using Lemma~\ref{lem:SteadyFourier} to bound the steady-state $\yer_{\rm st}$, formula \eqref{eq:erInitCond} to bound the transient response $\yer_{\xer_c}$, and the a priori error result from Assumption~\ref{assum:RomRequired} for the rest $\yer_{\rm rest}$.
\end{proof}
\begin{rem}\label{rem:FourOrderK}
The order $K$ of the Fourier series  $\uss^K$ is the only parameter to be chosen in our error bound. We propose $K \approx 10$ as a guide number. In Section~\ref{subsec:numCDPlayer} an illustrative study of its influence is made.
\end{rem}
Let us emphasize that our error bound allows for an efficient offline-online decomposition. Almost all required quantities except for the Fourier coefficients $\lambda_\ell$ are independent of the input $u$ and the initial conditions $\bv x_0$ and can therefore be determined in the offline phase. Moreover, it can be used that
$
\|u- \uss^K \|_{\funSpace{L}^2_T} = \sqrt{\|u \|_{\funSpace{L}^2_T}^2  - \| \uss^K \|_{\funSpace{L}^2_T}^2}
$
holds due to the orthogonality of $\uss^K$ and $u$. The evaluation of the Fourier coefficients is thus the main step of the online phase, and its computational cost does not scale with the dimension of the \acro{FOM}.

\section{Numerical validation} \label{sec:NumResults}
The efficiency of our error bound is showcased with two academic benchmark examples. We draw comparisons to the well-known a priori bound (Section~\ref{sec:BalRedResults}) and illustrate the influence of the order $K$ used in the underlying Fourier series, cf.\,Remark~\ref{rem:FourOrderK}. 

The \acro{FOM}s (Beam and CD Player) used in our numerical tests are from the \texttt{SLICOT} benchmark collection \cite[Section~24]{book:dimred2003}. All numerical results have been generated using \texttt{MATLAB} Version 9.1.0 (R2016b) on an Intel Core i5-7500 CPU with 16.0GB RAM. The simulations are based on the \texttt{MATLAB} built-in integrator \texttt{ode15s} with tolerances set to '\texttt{AbsTol} = $10^{-10}$' and '\texttt{RelTol}= $10^{-7}$'. Time integrals were approximated by quadrature with the trapezoidal rule on a uniform time mesh with $2\,000$ points. Finally, the simulation time $T= 2\pi$ is used for all tests.





\subsection{Beam (study in \acro{ROM} dimension $n$)} \label{subsec:numBeam}
The Beam benchmark is a single-input single-output model with $N = 349$. We simulate this \acro{FOM} with zero initial conditions and the input $u(t)=  4 \sin^3(2.7 \,t) + \e^{0.2\,t}$, $t \in [0,T]$, and compare it to simulations with \acro{ROM}s obtained by either \acro{BT} or \acro{SPA} and varying dimensions $n \in [2,30]$. As shown in  Fig.~\ref{fig:beam-ErrPlot}, our proposed bound (with parameter $K=10$) improves the error estimation of the a priori bound by about one order in average. The improvement is more pronounced for the reduction with \acro{SPA}, which also shows a smaller reduction error. In contrast to that, the a priori bound is the same for \acro{BT} and \acro{SPA}, i.e., not adapted to the setting.

\begin{figure}[tb]
{\center \hspace{0.2cm} \underline{Error vs. bounds (varying $n$)}}\\[0.2em]
\hspace{0.5cm} \underline{\acro{ROM} by \acro{BT}} \hspace{2.2cm} \underline{\acro{ROM} by \acro{SPA}} 
\includegraphics[height=3.0cm]{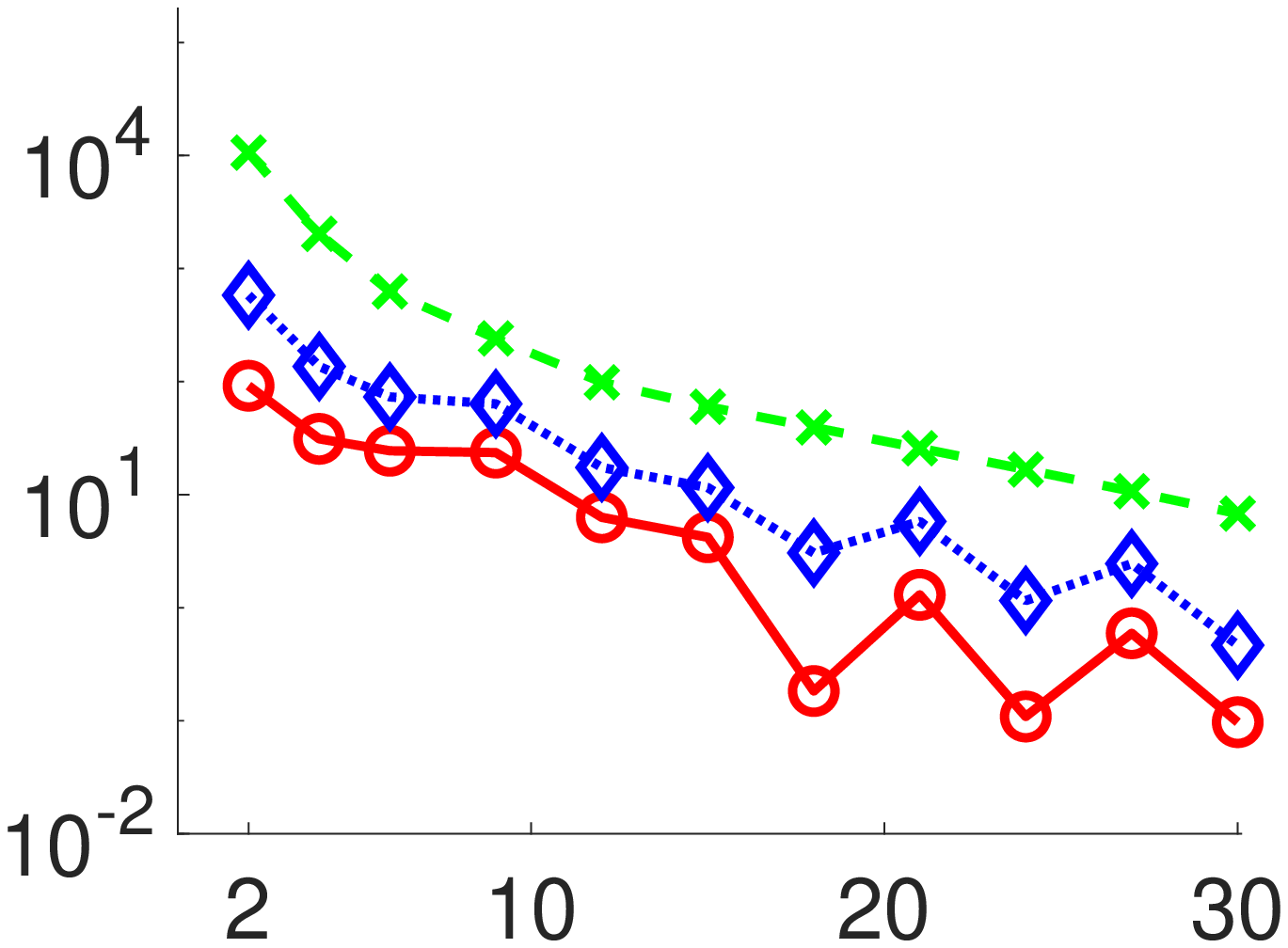}    \includegraphics[height=3.0cm]{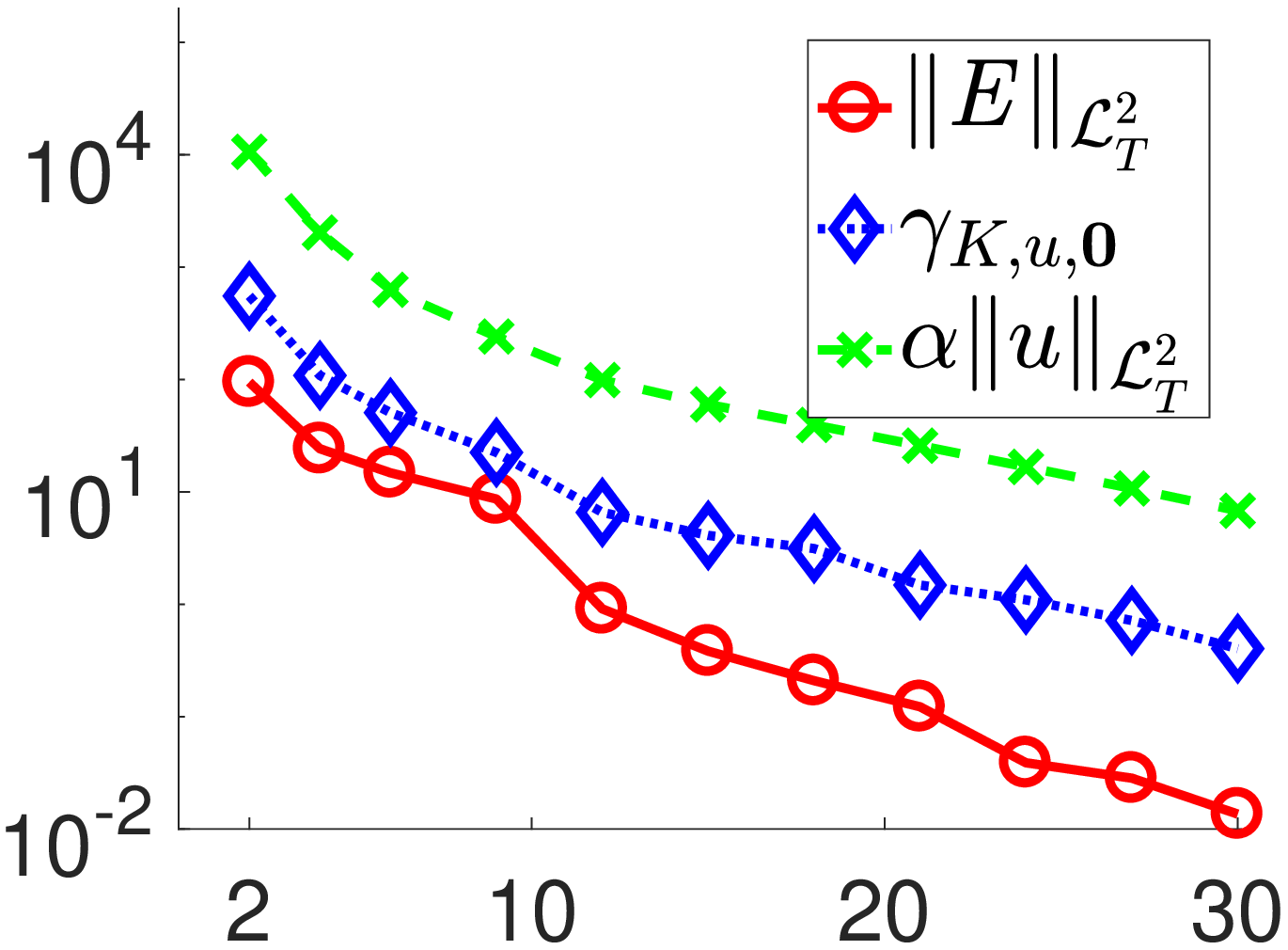}\\[0cm]
{\footnotesize \hspace{0.5cm} \acro{ROM} dimension $n$ \hspace{1.7cm}  \acro{ROM} dimension $n$}
\caption{Beam.  Reduction error versus the proposed a posteriori bound $\gamma_{K,u,\bv x_0}$ (with $\bv x_0 = \bv 0$ and $K=10$) and the standard a priori bound.
\label{fig:beam-ErrPlot}}                                                 
\end{figure}

\subsection{CD Player (study in order $K$ of Fourier series)} \label{subsec:numCDPlayer}
We consider the CD Player from the \acro{SLICOT} benchmark collection, which is a multi-input multi-output model with $N=120$. Since the paper is restricted to the single-input single-output case, we replace the input matrix by its first column and the output matrix by its second row for our numerical tests. This \acro{FOM} is reduced by \acro{BT} using $n=8$. We focus here on the influence of the parameter $K$ on the performance of our bound. For the comparisons, we adapt the a priori bound according to \cite{art:bls-InhomInit}, where errors related to initial conditions are bounded separately from the  input using \eqref{eq:erInitCond}, which yields the bound $\Delta_{\bv x_0} := \xer_0^T \Qer \xer_0$ with $\xer_0$ given as the initial conditions of the error system.

For the simulations, we choose $\bv x_0 = [1,\ldots,1]^T/100$, and the input
$$
u(t) = 
\begin{cases}
\hspace{0.4cm} t\hspace{0.5cm} - 8 \exp(-t/2) & t\leq \pi \\
(\pi-t) - 8 \exp(-t/2) & t> \pi
\end{cases},
$$
which has a discontinuity at $t = \pi$, see Fig.~\ref{fig:cdpla-input} for a plot in time. The nonzero initial conditions and the input discontinuity result in two peaks in the output response, cf.~Fig.~\ref{fig:cdpla-output}. 
\begin{figure}[tb]
\floatbox[{\capbeside\thisfloatsetup{capbesideposition={right,top},capbesidewidth=2.8cm}}]{figure}[\FBwidth]
{\caption{CD Player. Plot in time for used input $u$ and two Fourier series approximations of different order.
\label{fig:cdpla-input}
}}
{\includegraphics[height=3.5cm]{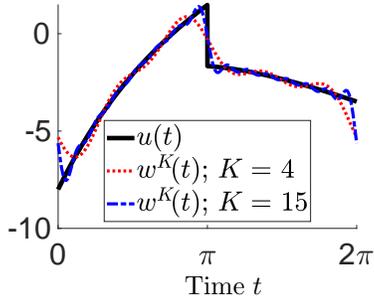}}     
{Time $t$\hspace{2.7cm}}               
\end{figure}
\begin{figure}[tb]
\floatbox[{\capbeside\thisfloatsetup{capbesideposition={right,top},capbesidewidth=2.8cm}}]{figure}[\FBwidth]
{\caption{CD Player. Plot in time of outputs $y$ and $\rom{y}$ for the \acro{FOM} and \acro{ROM}, respectively.
\label{fig:cdpla-output}
}}
{\includegraphics[height=3.5cm]{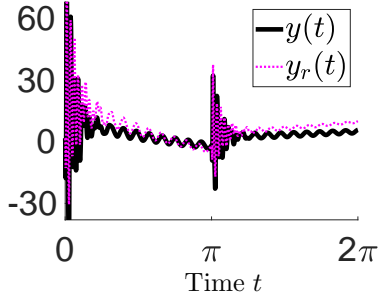}}
{Time $t$\hspace{2.7cm}}                      
\end{figure}
\begin{figure}[tb]
\underline{Fourier series quality} \hspace{1.1cm} \underline{Error vs. bounds}\\[0.2em]
\begin{tabular}{l|l}
\includegraphics[height=2.8cm]{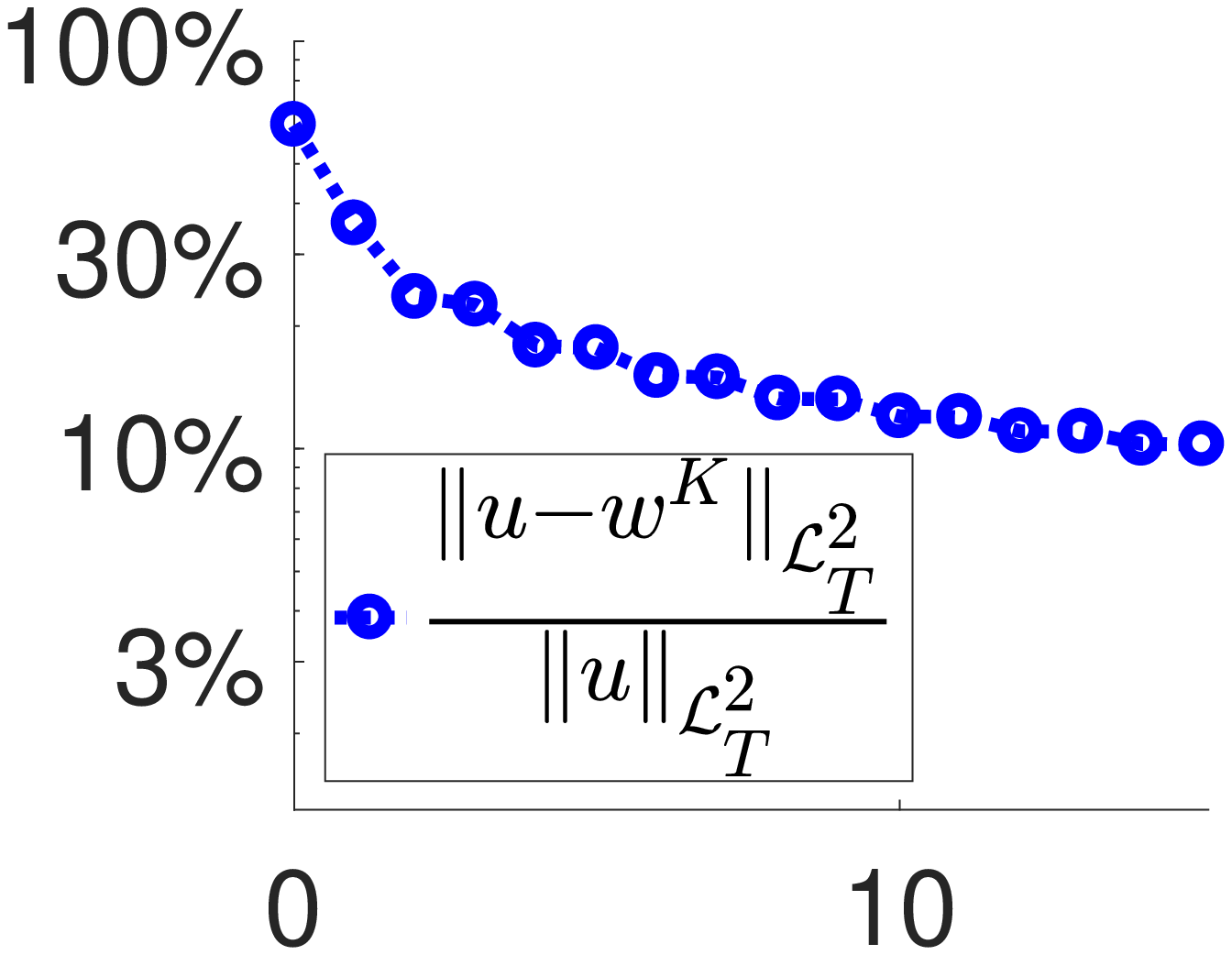} \hspace{0.2cm}  
&
  \includegraphics[height=2.8cm]{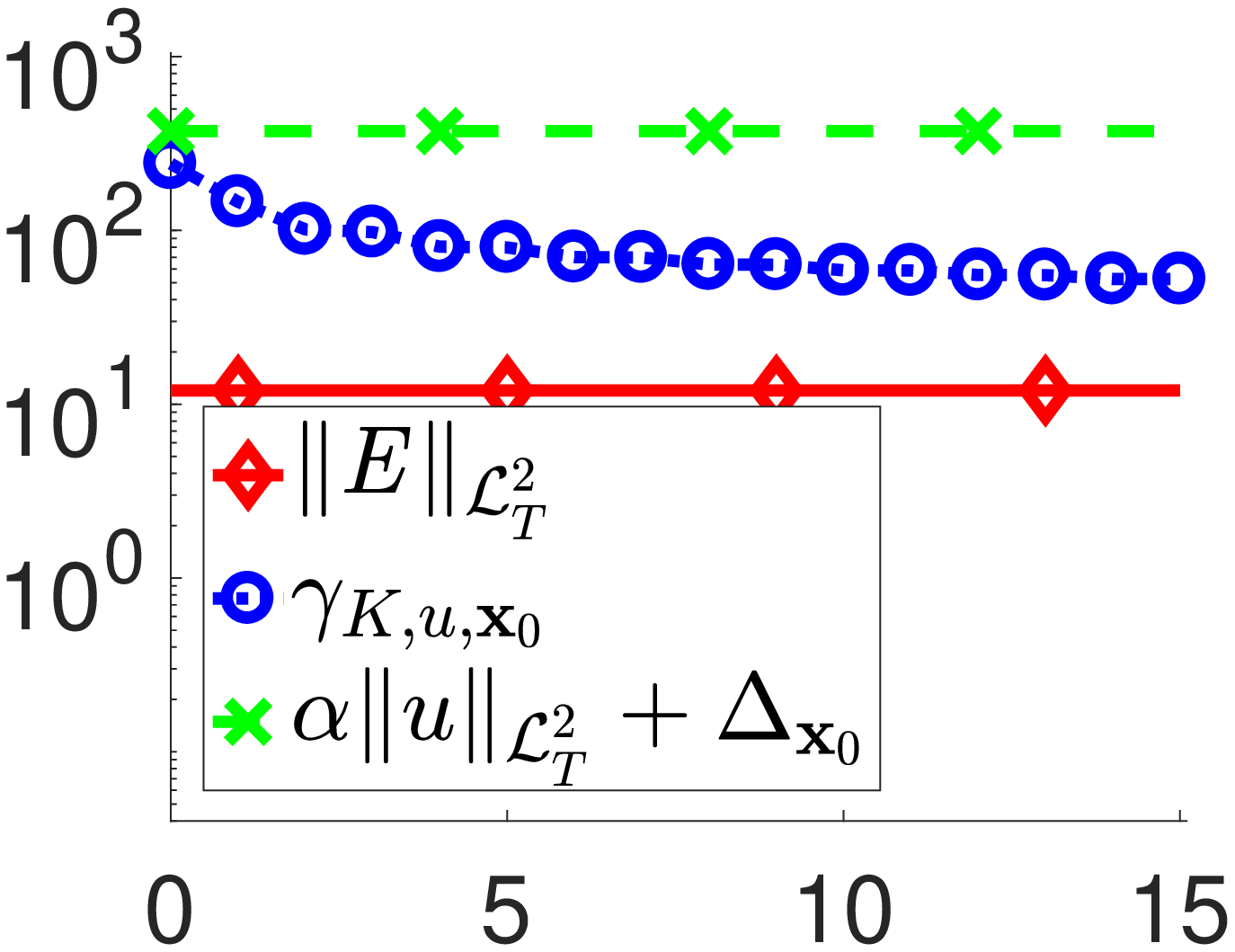}
\end{tabular}\\[0.0em]
{\footnotesize \hspace{1.2cm}Order $K$ of $\uss^K$ \hspace{1.7cm}  Order $K$ of $\uss^K$}
\caption{CD Player, parameter study in $K$. \textit{Left:} Relative error of Fourier series approximation. \textit{Right:} Reduction error versus the proposed a posteriori bound $\gamma_{K,u,\bv x_0}$ and the standard a priori bound (using $\Delta_{\bv x_0}$ as in \cite{art:bls-InhomInit} for the initial conditions). \acro{ROM} obtained by \acro{BT} with $n=8$.
\label{fig:cdpla-ErrorPlot}}   
\end{figure}
Moreover, the discontinuity implies a slow decay of the Fourier series approximation of $u$.  This is illustrated in Fig.~\ref{fig:cdpla-ErrorPlot} for $K\in [0,15]$, where also the resulting a posteriori bound in comparison to the a priori bound and the reduction error is shown. It is seen that the Fourier series approximation $\uss^K$ has still an error of about $10\%$ for  $K=15$. The error is overestimated by less than one order by our a posteriori error bound with $K\geq 10$, which is a significant improvement compared to the almost two order of magnitude observed for the a priori bound.
This small example showcases that our a posteriori bound is still effective even if the underlying Fourier series approximation only has a mediocre approximation quality.


\section*{Conclusion}
We proposed an a posteriori extension of the well-konwn a priori error bound for balancing-related model reduction. It alleviates the worst-case type error estimation that is inherent in the a priori error bound. The idea is to filter out an approximation on the input signal, for which a more explicit error analysis can be done in an efficient manner. In this paper, we used a truncated Fourier series approximation of the input and derived an efficient offline-online decomposition for it.

On a final note, we would like to point towards possible extensions of our result.
For certain applications, it could be interesting to study other input approximations, e.g., using signal generators related to a known frequency range of interest. The extension of our bound for time-limited balanced truncation is straight forward, using the results from \cite{art:redmann2020lt2}, cf., Remark~\ref{rem:timeLimited}. Multi-input multi-output systems could also be considered, but this requires a separate approximation of each of the input components.

\bibliographystyle{abbrv}   
\bibliography{ssbound}

\begin{thebibliography}{10}

\bibitem{book:antoulas2005}
A.~Antoulas.
\newblock {\em Approximation of Large-Scale Dynamical Systems}, volume~6 of
  {\em Adv. Des. Control}.
\newblock {SIAM}, Philadelphia, PA, 2005.

\bibitem{art:steadystate-mor-2010}
A.~Astolfi.
\newblock Model reduction by moment matching for linear and nonlinear systems.
\newblock {\em IEEE Transactions on Automatic Control}, 55(10):2321--2336,
  2010.

\bibitem{book:dimred2003}
P.~Benner, V.~Mehrmann, and D.~C. Sorensen, editors.
\newblock {\em Dimension Reduction of Large-Scale Systems}.
\newblock Lecture Notes in Computational Science and Engineering. Springer, 1
  edition, 2005.

\bibitem{feng2017some}
L.~Feng, A.~C. Antoulas, and P.~Benner.
\newblock Some a posteriori error bounds for reduced-order modelling of (non-)
  parametrized linear systems.
\newblock {\em ESAIM: Math. Model. Numer. Anal.}, 51(6):2127--2158, 2017.

\bibitem{feng2019new}
L.~Feng and P.~Benner.
\newblock A new error estimator for reduced-order modeling of linear parametric
  systems.
\newblock {\em IEEE Transactions on Microwave Theory and Techniques},
  67(12):4848--4859, 2019.

\bibitem{book:green2012linear}
M.~Green and D.~J. Limebeer.
\newblock {\em Linear robust control}.
\newblock CRRC, 2012.

\bibitem{haasdonk2011efficient}
B.~Haasdonk and M.~Ohlberger.
\newblock Efficient reduced models and a posteriori error estimation for
  parametrized dynamical systems by offline/online decomposition.
\newblock {\em Math. Comput. Model. Dyn. Syst.}, 17(2):145--161, 2011.

\bibitem{art:isidori2008steady}
A.~Isidori and C.~I. Byrnes.
\newblock Steady-state behaviors in nonlinear systems with an application to
  robust disturbance rejection.
\newblock {\em Annu. Rev. Control}, 32(1):1--16, 2008.

\bibitem{art:bls-InhomInit}
B.~Liljegren-Sailer.
\newblock Effective error estimation for model reduction with inhomogeneous
  initial conditions.
\newblock arXiv e-prints 2201.06631, 2022.

\bibitem{art:SPA-liu1989}
Y.~Liu and B.~D. Anderson.
\newblock Singular perturbation approximation of balanced systems.
\newblock {\em Internat. J. Control}, 50(4):1379--1405, 1989.

\bibitem{art:BT-moore1981}
B.~Moore.
\newblock Principal component analysis in linear systems: Controllability,
  observability, and model reduction.
\newblock {\em IEEE Trans. Automat. Control}, 26(1):17--32, 1981.

\bibitem{art:redmann2020lt2}
M.~Redmann.
\newblock An {$L_T^2$}-error bound for time-limited balanced truncation.
\newblock {\em Systems Control Lett.}, 136:104620, 2020.

\end{thebibliography}


\end{document}